\documentclass[oneside,12pt]{amsart}
\usepackage{eurosym}
\usepackage{amsfonts} 
\usepackage{amsmath,amssymb,amsthm,xcolor}
\usepackage{latexsym}
\usepackage{float}
\usepackage{placeins}
\usepackage{bm}
\usepackage[utf8]{inputenc}
\usepackage[british]{babel}
\usepackage{a4wide}
\usepackage{pdfsync}
\usepackage[colorlinks=true,linkcolor=blue,citecolor=blue]{hyperref}
\usepackage[all]{xy}
\usepackage{parskip}
\usepackage{multirow}
\usepackage{array}
\usepackage[toc,page]{appendix}
\usepackage{graphicx}
\usepackage{comment}
\usepackage[mathscr]{euscript}
\usepackage{tikz-cd}

\theoremstyle{definition}

\newtheorem{lemma}{Lemma}[section]
\newtheorem{theorem}[lemma]{Theorem}
\newtheorem{proposition}[lemma]{Proposition}

\newtheorem*{theorem*}{Theorem}
\newtheorem*{definition*}{Definition}

\numberwithin{equation}{section}

\begin{document}

\title[Stability of Einstein 4-manifolds satisfying a chiral curvature condition]{Stability of Einstein 4-manifolds satisfying a chiral curvature condition}

\author{Diego Artacho}
\address{D.~Artacho: Department of Mathematics, KU Leuven, Celestijnenlaan 200B, 3001 Leuven, Belgium}
\email{diego.artachodeobeso@kuleuven.be}

\begin{abstract}
Let $(M,g)$ be a compact oriented Einstein four-manifold with Einstein constant $E$ and let $\widehat{R}^+$ denote the action of the Riemann curvature tensor on self-dual two-forms. We show that if $\widehat{R}^+ < 0$, then $g$ is strictly linearly stable for the Einstein--Hilbert functional, thus giving a chiral criterion for stability. Our result is stronger than the one by Fine--Krasnov--Singer, who conclude local rigidity from $\widehat{R}^+ < 0$ by proving stability for a different action functional. Our proof proceeds by showing that $(M,g)$ admits a natural spin$^h$ structure carrying a non-zero parallel spin$^h$-spinor. We then apply a lower bound on the Lichnerowicz Laplacian on traceless symmetric two-tensors in the presence of such a spinor. 
\end{abstract}

\maketitle

\section{Introduction}

Let $(M,g)$ be a compact Einstein $n$-manifold with Einstein constant $E$. A classical problem in Riemannian geometry concerns the study of the moduli space of Einstein metrics of unit volume on $M$ modulo the action of its diffeomorphism group. Semmelmann--Schwahn \cite{SS} wrote a survey on the history and state of the art in this area. One way to study the local structure of this moduli space is to note that Einstein metrics are precisely the critical points of the Einstein-Hilbert functional, and then study its second variation at $g$. The metric $g$ is said to be \emph{strictly linearly stable} if $\Delta_L - 2E > 0$ on the space of transverse traceless symmetric two-tensors (tt-tensors), where $\Delta_L$ denotes the Lichnerowicz Laplacian -- see \cite{SS} for details. 

The first stability results were obtained by imposing restrictions on the Riemann curvature tensor $R$. This tensor acts on traceless symmetric two-tensors $h$ by 
\[
(\mathring{R}h)_{ij} = R_{iklj} h_{kl}
\]
in terms of a local orthonormal frame, where summation is taken over repeated indices. A classical result by Koiso \cite{K78} states that if $\mathring{R} < \max\{-E , E/2\}$ then $g$ is strictly linearly stable. 

The Riemann curvature tensor also acts on two-forms $\omega$ by 
\[
(\widehat{R} \omega)_{ij} = -\frac{1}{2} R_{ijkl} \omega_{kl} . 
\]
If $n=4$ and $M$ is oriented, the space of two-forms decomposes into $\pm 1$-eigenspaces of the Hodge star operator as $\Lambda^2 T^*M = \Lambda^+ \oplus \Lambda^-$. If, moreover, $g$ is Einstein, then $\widehat{R}$ preserves this decomposition -- see Section \ref{sec:curvature}. We denote by $\widehat{R}^{\pm}$ the restriction of $\widehat{R}$ to $\Lambda^{\pm}$. The main result of this paper guarantees strict linear stability of $g$ by imposing restrictions on $\widehat{R}^+$ (or $\widehat{R}^-$) alone: 

\begin{theorem}\label{thm:main}
Let $(M,g)$ be a compact oriented Einstein four-manifold. If $\widehat{R}^+ < 0$ or $\widehat{R}^- < 0$, then $g$ is strictly linearly stable. 
\end{theorem}

Here, $\widehat{R}^+ < 0$ means that for every $\omega \in \Lambda^+$ one has $\langle \widehat{R}^+ \omega , \omega \rangle < 0$. This should be regarded as a hypothesis on \emph{half of the full Riemann curvature tensor}, as opposed to Koiso's criterion, which involves all of it through $\mathring{R}$. In a recent paper, Fine--Krasnov--Singer \cite{FKS} showed that, if $\widehat{R}^+<0$, then $g$ is locally rigid. Their proof proceeds by establishing strict linear stability for a functional different from the Einstein–Hilbert functional. We obtain a stronger result, namely strict linear stability for the usual Einstein–Hilbert functional, which in particular implies local rigidity \cite{DWW05,DK24}. 

Our proof makes use of a recent bound of the Lichnerowicz Laplacian acting on traceless symmetric two-tensors obtained by the author \cite{A25} in terms of a parallel twisted spinor: 

\begin{theorem}[\cite{A25}]\label{thm:criterion}
Let $(M,g)$ be a compact and oriented Einstein $n$-manifold with Einstein constant $E$, and suppose that there exists an oriented Euclidean vector bundle $F \to M$ equipped with a metric connection $\nabla^F$ such that 
\begin{itemize}
    \item $TM \oplus F$ is spin, and so the locally defined spinor bundles $\Sigma M$ and $\Sigma F$ of $TM$ and $F$ respectively define an honest vector bundle $\Sigma M \otimes_{\mathbb{C}} \Sigma F$ over $M$; and
    \item $\Sigma M \otimes_{\mathbb{C}} \Sigma F$ has a non-zero parallel section $\psi$ with respect to the connection induced by the Levi-Civita connection of $g$ and $\nabla^F$. 
\end{itemize}
Then, the Lichnerowicz Laplacian $\Delta_{\mathrm{L}}$ on traceless symmetric two-tensors satisfies 
    \[
    \left\langle \Delta_{\mathrm{L}} h , h \right\rangle \geq E \langle h , h \rangle - \left \langle \mathcal{R}_{\psi} h , h \right \rangle ,
    \]
    where $\mathcal{R}_{\psi} \colon \mathrm{Sym}_0^2 T^*M \to T^*M \otimes_{\mathbb{R}} T^*M$ is defined with respect to a local orthonormal frame $(e_1 , \dots , e_n)$ as 
    \begin{equation}\label{eq:rpsi_def}
    (\mathcal{R}_{\psi} h)_{pi} = -\frac{1}{2} h_{qj} R_{pjkl} \left\langle e_q \cdot e_i \cdot e_k \cdot e_l \cdot \psi , \psi \right\rangle , 
    \end{equation}
    where $\cdot$ denotes Clifford multiplication and $\left\langle \cdot , \cdot \right\rangle$ is the standard Hermitian product on the spinor bundle induced by the metrics on $TM$ and $F$.
    
    In particular, if $\mathcal{R}_{\psi} < -E$, then $g$ is strictly linearly stable. \qed
\end{theorem} 

Theorem \ref{thm:criterion} is a consequence of an extension of the spin \cite{DWW05} and spin$^c$ \cite{DWW07} methods of Dai--Wang--Wei to higher order spin$^r$ structures \cite{EH16}. Indeed, in the language of spin$^r$ structures, the hypotheses of Theorem \ref{thm:criterion} can be rephrased as $(M,g)$ admitting a spin$^r$ structure with a non-zero parallel spin$^r$-spinor, where $r$ is the rank of $F$. In this paper, we shall only need the case $r = 3$, which corresponds to spin$^h$ structures. 

The proof of Theorem \ref{thm:main} has four steps. For a compact and oriented Einstein four-manifold $(M,g)$, we 
\begin{itemize}
\item observe that the space of traceless symmetric two-tensors of an oriented Riemannian four-manifolds is isomorphic to $\Lambda^+ \otimes \Lambda^-$ -- Section \ref{sec:curvature}; 
\item show that $(M,g)$ admits a canonical spin$^3$ structure carrying a non-zero parallel spin$^3$-spinor $\psi$ -- Proposition \ref{prop:spin}; 
\item express the operator $\mathcal{R}_{\psi}$ in terms of $\widehat{R}^{\pm}$ -- equation \eqref{eq:rpsi}; 
\item apply Theorem \ref{thm:criterion}. 
\end{itemize}

As a final remark, spin$^3$ structures are also known as spin$^h$ or spin$^q$ structures in the literature and have been the subject of extensive research in recent years -- see e.g. \cite{L23} and references therein. Every spin structure induces a spin$^c$ structure, and every spin$^c$ structure induces a spin$^h$ structure. It is known that every oriented four-manifold admits a spin$^c$ structure \cite{HH,TV05}, but there need not exist a parallel spin$^c$-spinor \cite{M97}. By contrast, Proposition \ref{prop:spin} shows that every oriented Riemannian four-manifold admits a natural choice of spin$^h$ structure and non-zero parallel spin$^h$-spinor. This gives Proposition \ref{prop:spin} independent geometric significance. In particular, it shows that \cite[Thm. 3.1]{HH07} cannot hold as stated, because the round sphere $S^4$ is simply connected, irreducible, non-Ricci-flat, and non-K{\"a}hler, yet Proposition \ref{prop:spin} equips it with a non-zero parallel spin$^q$ spinor.

\section{Curvature operators in dimension four} \label{sec:curvature}

Let $(M,g)$ be an oriented Riemannian four-manifold. Its Riemann curvature tensor is defined by 
\[
R(X,Y)Z = \nabla_X \nabla_Y Z - \nabla_Y \nabla_X Z - \nabla_{[X,Y]} Z . 
\]
In a local orthonormal frame $(e_1 , e_2 , e_3 , e_4)$, we write 
\[
R_{ijkl} = g(R(e_i , e_j) e_k , e_l) .  
\]
Similarly, given a two-form $\omega$ or a symmetric two-tensor $h$ we write 
\[
\omega_{ij} = \omega(e_i , e_j) , \qquad h_{ij} = h(e_i , e_j) . 
\]
The Riemann curvature tensor acts on two-forms $\omega$ and traceless symmetric two-tensors $h$ by 
\[ 
(\widehat{R} \omega)_{ij} = -\frac{1}{2} R_{ijkl} \omega_{kl} , \qquad (\mathring{R}h)_{ij} = R_{iklj}h_{kl} .  
\]
These are referred to as the \emph{curvature operators of the first and second kind} respectively. Our conventions for $R$ are such that our $\widehat{R}$ on two-forms coincides with that of \cite{FKS, CGT, SS}, and our $\mathring{R}$ on traceless symmetric two-tensors coincides with that of \cite{SS} but differs from that of \cite{CGT} by a sign. In a space of constant sectional curvature $K$, our definitions yield $\widehat{R} = K \, \mathrm{Id}$ and $\mathring{R} = -K \, \mathrm{Id}$. 

A special feature of dimension four is that the Hodge star operator $\star$ defines a linear isomorphism on two-forms that squares to the identity, and hence induces a decomposition 
\[
\Lambda^2 T^*M = \Lambda^+ \oplus \Lambda^- 
\]
into $\pm 1$-eigenspaces of $\star$. It is easy to see that 
\[
\omega \in \Lambda^{\pm} \quad \text{if and only if} \quad \omega_{ij} = \pm \frac{1}{2} \varepsilon_{ijkl} \omega_{kl} ,
\] 
where $\varepsilon_{ijkl}$ is the sign of the permutation $(i,j,k,l)$ of $(1,2,3,4)$. 

Another way of seeing that dimension four is special is that the Lie algebra $\mathfrak{so}(4)$ is the only reducible one among the special orthogonal algebras. In fact, if $\lambda_n \colon \mathrm{Spin}(n) \to \mathrm{SO}(n)$ is the usual double cover, we have that $\mathrm{Spin}(4) \cong \mathrm{Spin}(3) \times \mathrm{Spin}(3)$, and denoting by $\mathrm{pr}_{\pm}$ the projection to the first and second component respectively we have a commutative diagram of Lie groups and homomorphisms 
\[
\begin{tikzcd}
\mathrm{Spin}(4) \arrow[r, "\mathrm{pr}_{\pm}"] \arrow[d, "\lambda_4"'] & \mathrm{Spin}(3) \arrow[d, "\lambda_3"] \\
\mathrm{SO}(4) \arrow[r, "\widehat{\mathrm{pr}}_{\pm}"']                 & \mathrm{SO}(3)
\end{tikzcd}
\]

If $\mathrm{Fr} \to M$ denotes the principal $\mathrm{SO}(4)$-bundle of oriented orthonormal frames of $(M,g)$, then to any representation $\rho$ of $\mathrm{SO}(4)$ corresponds the associated vector bundle $\mathrm{Ass}(\mathrm{Fr} ,\rho)$, which comes equipped with a covariant derivative induced from the Levi-Civita connection on $\mathrm{Fr}$. For example, $TM = \mathrm{Ass}(\mathrm{Fr},\mathrm{Id}_{\mathrm{SO}(4)})$. Moreover, $\Lambda^{\pm} = \mathrm{Ass}(\mathrm{Fr}, \widehat{\mathrm{pr}}_{\pm})$.

If $g$ is an Einstein metric, the operator $\widehat{R}$ respects the decomposition $\Lambda^2 T^*M = \Lambda^+ \oplus \Lambda^-$, i.e., $\widehat{R}(\Lambda^{\pm}) \subseteq \Lambda^{\pm}$. We denote by $\widehat{R}^{\pm}$ the restriction of $\widehat{R}$ to $\Lambda^{\pm}$. 

Another useful fact is that the space of traceless symmetric two-tensors $\mathrm{Sym}_0^2 T^*M$ is isomorphic to $\Lambda^+ \otimes_{\mathbb{R}} \Lambda^-$ via 
\begin{equation}\label{eq:sym_lambda}
(\alpha \otimes \beta)_{ij} = \alpha_{ik} \beta_{jk} . 
\end{equation}

In light of this isomorphism, it is natural to ask whether the action of the operator $\mathring{R}$ can be expressed in terms of $\widehat{R}$. Although a more general result could be proved, we will only need the following:  

\begin{proposition}[{\cite[Thm. 1.8]{CGT}}] \label{prop:rcirc}
    Let $(M,g)$ be a compact and oriented Einstein four-manifold with Einstein constant $E$. Then, 
    \[
    \mathring{R} =  (\widehat{R}^+ \otimes \mathrm{Id}_{\Lambda^-}) + (\mathrm{Id}_{\Lambda^+} \otimes \widehat{R}^-) - E \, \mathrm{Id}_{\Lambda^+ \otimes \Lambda^-} . \hfill 
    \] \qed
\end{proposition}

Recall that, as we mentioned above, the convention used in \cite{CGT} makes their $\widehat{R}$ coincide with ours, whereas their $\mathring{R}$ differs from ours by a sign.

\section{A special twisted spinor} \label{sec:spin}

The aim of this section is to prove that every oriented four-manifold comes equipped with a natural parallel spin$^h$ spinor. Specifically, we will prove the following result: 

\begin{proposition}\label{prop:spin}
    Let $(M,g)$ be an oriented Riemannian four-manifold. Then, the following hold: 
    \begin{enumerate}
        \item $TM \oplus \Lambda^+$ is a spin vector bundle, and the locally defined spinor bundles $\Sigma M = \Sigma^+ M \oplus \Sigma^- M$ of $M$ and $\Sigma \Lambda^+$ of $\Lambda^+$ define an honest spinor bundle $\Sigma M \otimes_{\mathbb{C}} \Sigma \Lambda^+$over $M$; 
        \item $\Sigma^+M \otimes_{\mathbb{C}} \Sigma \Lambda^+ \cong \Sigma^+M \otimes_{\mathbb{C}} \Sigma^+M \cong \mathrm{End}_{\mathbb{C}}(\Sigma^+M)$; and 
        \item $\mathrm{Id}_{\Sigma^+M} \in \Gamma(\mathrm{End}(\Sigma^+M))$ is parallel with respect to the connection induced by the Levi-Civita connection. 
    \end{enumerate}
\end{proposition}

\begin{proof}
    $TM$ and $\Lambda^+$ are associated to the representations $\mathrm{Id}_{\mathrm{SO}(4)}$ and $\widehat{\mathrm{pr}}_+$ of $\mathrm{SO}(4)$ respectively. Hence, locally, their spinor bundles are associated to the representations $\Delta_4 = \Delta^+_4 \oplus \Delta^-_4$ and $\Delta_3 \circ \mathrm{pr}_+$ of $\mathrm{Spin}(4)$ respectively, where $\Delta_n$ denotes the spin representation of $\mathrm{Spin}(n)$. However, note that $\Delta^+_4 \otimes (\Delta_3 \circ \mathrm{pr}_+)$ defines a well-defined representation of $\mathrm{SO}(4)$, and hence its associated bundle is an honest vector bundle over $M$. This proves point (1). Point (2) follows from the classical representation-theoretic facts $\Delta_3 \cong \Delta^+_4 \cong(\Delta^+_4)^*$. Finally, it is clear that, for any vector bundle equipped with a connection, the identity section of its endomorphism bundle is parallel with respect to the connection induced on the endomorphism bundle. This proves point (3). 
\end{proof}

In the language of spin$^r$ geometry \cite{EH16}, $\Lambda^+$ equipped with the connection induced by the Levi-Civita connection is the auxiliary bundle of a spin$^h$ structure on $(M,g)$, and its twisted spinor bundle admits a natural parallel section. This spinor will be the key ingredient for the proof of Theorem \ref{thm:main}.  

We can make this spinor explicit by choosing a model for the spin representation. Following \cite{AHL}, let $(e_1 , e_2 , e_3 , e_4)$ be an oriented local orthonormal frame, and define
\begin{align*} 
L := \mathrm{span}_{\mathbb{C}} \left\{ x_1 := \frac{1}{\sqrt{2}} \left( e_1 - i e_2 \right) , \, x_2 := \frac{1}{\sqrt{2}} \left( e_3 - i e_4 \right) \right\} \, , \\
L' := \mathrm{span}_{\mathbb{C}} \left\{ y_1 := \frac{1}{\sqrt{2}} \left( e_1 + i e_2 \right) , \, y_2 := \frac{1}{\sqrt{2}} \left( e_3 + i e_4 \right) \right\} \, . 
\end{align*} 
Note that the spin representation of the complex Clifford algebra $\mathbb{C}l(4)$ (with the convention that $v \cdot v = - \langle v , v \rangle$) is isomorphic to the action on the exterior algebra $\Lambda^{\bullet} L'$ defined by extending 
\begin{equation}\label{eq:spinrep}
x_j \cdot \eta := i \sqrt{2} x_j \lrcorner \eta \, , \quad y_j \cdot \eta := i \sqrt{2} y_j \wedge \eta \, ,  
\end{equation} 
for $i = 1,2$ and $\eta \in \Lambda^{\bullet} L'$. The spin representation $\Delta_4$ is the restriction of this representation to $\mathrm{Spin}(4) \subset \mathbb{C}l(4)$, and moreover $\Delta^+ = \Lambda^{\bullet}_{\text{even}} L'$ and $\Delta^- = \Lambda^{\bullet}_{\text{odd}} L'$. 

Using this model, the spinor given by Proposition \ref{prop:spin} is given by 
\begin{equation}\label{eq:psi}
\psi = \frac{1}{\sqrt{2}} \left( 1 \otimes 1 + (y_1 \wedge y_2) \otimes (y_1 \wedge y_2) \right) . 
\end{equation}

It is now easy to check using \eqref{eq:spinrep} that $\psi$ is annihilated by the action of every $e_i \cdot e_j \in \mathfrak{spin}(4)$ with $i < j$, and thus defines a parallel section of $\Sigma^+ M \otimes_{\mathbb{C}} \Sigma^+ M$. For the sake of clarity, and because we will need to do calculations in Section \ref{sec:final}, we show that $e_1 \cdot e_3 \in \mathfrak{spin}(4)$ annihilates $\psi$. Indeed, 
\begin{align*}
    \sqrt{2} \, e_1 e_3 \cdot \psi &= (e_1 e_3 \cdot 1) \otimes 1 + 1 \otimes (e_1 e_3 \cdot 1) + (e_1 e_3 \cdot y_1 \wedge y_2) \otimes y_1 \wedge y_2 + y_1 \wedge y_2 \otimes (e_1 e_3 \cdot y_1 \wedge y_2) . 
\end{align*}

But 
\[
e_1 e_3 \cdot 1 = e_1 \cdot (i y_2) = - y_1 \wedge y_2 \quad \text{and} \quad e_1 e_3 \cdot y_1 \wedge y_2 = e_1 \cdot (-i y_1) = 1 ,  
\]
and hence $e_1 e_3 \cdot \psi = 0$. 

Later we will need the following Lemma, whose proof is a straightforward calculation using \eqref{eq:spinrep}: 

\begin{lemma}\label{lem:psi}
    Let $q,i,k,l \in \{1,2,3,4\}$ and $\psi$ be given by \eqref{eq:psi}. Then, 
    \[
    \left\langle e_q \cdot e_i \cdot e_k \cdot e_l \cdot \psi , \psi \right\rangle = - \varepsilon_{qikl} + \delta_{qi}\delta_{kl} - \delta_{qk}\delta_{il} + \delta_{ql}\delta_{ik} ,  
    \]
    where $\delta$ is Kronecker's delta and $\varepsilon_{qikl}$ is the sign of the permutation $(q,i,k,l)$ of $(1,2,3,4)$. \qed
\end{lemma}

\section{Proof of the main theorem} \label{sec:final}

Let $(M,g)$ be a compact and oriented Einstein four-manifold with Einstein constant $E$, and let $\psi$ be the spinor given by Proposition \ref{prop:spin} normalised to unit length. As in \cite{A25}, $\psi$ determines a bundle map 
\[
\mathcal{R}_{\psi} \colon \mathrm{Sym}^2_0 T^*M \to T^*M  \otimes_{\mathbb{R}} T^*M
\]
that is given by \eqref{eq:rpsi_def} in terms of a local orthonormal frame. 

We will show that, under the identification \eqref{eq:sym_lambda} of $\mathrm{Sym}_0^2 T^*M$ and $\Lambda^+ \otimes \Lambda^-$,
\begin{equation}\label{eq:rpsi}
    \mathcal{R}_{\psi} = (2\widehat{R}^+ - E \, \mathrm{Id_{\Lambda^+}}) \otimes \mathrm{Id}_{\Lambda^-}     .  
\end{equation}
Hence, applying Theorem \ref{thm:criterion}, we will finish the proof. Note that reversing the orientation swaps $\widehat{R^+}$ and $\widehat{R^-}$. 

Let $h \in \mathrm{Sym}_0^2 T^*M$. Then, by \eqref{eq:rpsi_def}
\[
\mathcal{R}_{\psi}(h)_{ip} = -\frac{1}{2} h_{qj} R_{pjkl} \left\langle e_q e_i e_k e_l \cdot \psi , \psi \right\rangle . 
\]

By Lemma \ref{lem:psi} and the usual symmetries of $R$, 
\[
\mathcal{R}_{\psi}(h)_{ip} = \frac{1}{2} h_{qj} \varepsilon_{qikl} R_{pjkl} + h_{qj} R_{pjqi} = \frac{1}{2} h_{qj} \varepsilon_{qikl} R_{pjkl} + (\mathring{R}h)_{ip} . 
\]

Choose $h = \alpha \otimes \beta$ for $\alpha \in \Lambda^+$ and $\beta \in \Lambda^-$, and use self-duality of $\alpha$  
\[
\alpha_{qr} = \frac{1}{2} \varepsilon_{qrmn} \alpha_{mn}
\]
and the classical identity 
\[
\varepsilon_{qrmn} \varepsilon_{qikl} = \delta_{ri}\delta_{mk}\delta_{nl}+\delta_{rk}\delta_{ml}\delta_{ni}+\delta_{rl}\delta_{mi}\delta_{nk} - \delta_{ri}\delta_{ml}\delta_{nk}-\delta_{rk}\delta_{mi}\delta_{nl}
-\delta_{rl}\delta_{mk}\delta_{ni}
\]
together with the usual symmetries of the Riemann curvature tensor to calculate: 
\begin{align*}
    \frac{1}{2} h_{qj} \varepsilon_{qikl} R_{pjkl} &= \frac{1}{2} \alpha_{qr} \beta_{jr} \varepsilon_{qikl} R_{pjkl} = \frac{1}{4} \alpha_{mn} \beta_{jr} R_{pjkl} \varepsilon_{qrmn} \varepsilon_{qikl} \\
    &= \frac{1}{2} \alpha_{kl} \beta_{ji} R_{pjkl} + \alpha_{li} \beta_{jk} R_{pjkl} . 
\end{align*}

By the First Bianchi Identity and relabelling the indices, 
\begin{align*}
    \alpha_{li} \beta_{jk} R_{pjkl} &= - \alpha_{li} \beta_{jk} R_{pklj} - \alpha_{li} \beta_{jk} R_{pljk} = - \alpha_{li} \beta_{jk} R_{pjkl} - \alpha_{li} \beta_{jk} R_{pljk} . 
\end{align*}

Hence, 
\[
\alpha_{li} \beta_{jk} R_{pjkl} = - \frac{1}{2} \alpha_{li} \beta_{jk} R_{pljk}
\]
and thus 
\begin{align*}
\frac{1}{2} h_{qj} \varepsilon_{qikl} R_{pjkl} &= \frac{1}{2} \alpha_{kl} \beta_{ji} R_{pjkl} - \frac{1}{2} \alpha_{li} \beta_{jk} R_{pljk} \\
&= -(\widehat{R}^+ \alpha)_{pj} \beta_{ji} + \alpha_{li} (\widehat{R}^- \beta)_{pl} \\
&= ((\widehat{R}^+ \alpha) \otimes \beta)_{ip} - (\alpha \otimes (\widehat{R}^- \beta))_{ip} . 
\end{align*}

Putting everything together and using Proposition \ref{prop:rcirc} finishes the proof of equation \eqref{eq:rpsi} and hence of Theorem \ref{thm:main}. 

\section*{Acknowledgements}

 This work has been partially supported by FWO and FNRS under EOS project G0I2222N.

\bibliographystyle{alphaurl}
\bibliography{references.bib}

\end{document}